\documentclass[leqno]{amsart}
\usepackage{latexsym,amsmath,amsfonts,amssymb,amsthm}

\theoremstyle{definition}

\newtheorem{exmp}{Example}

\theoremstyle{plain}
\newtheorem{thm}{Theorem}
\newtheorem{prop}{Proposition}
\newtheorem{lem}{Lemma}
\newtheorem{cor}{Corollary}

\theoremstyle{remark}
\newtheorem*{rem}{Remark}

\newcounter{claim}

\DeclareMathOperator{\im}{Im}

\title[Holomorphic mappings of once-holed tori]
{On sets of marked once-holed tori allowing \\
holomorphic mappings into \\
Riemann surfaces with marked handle}
\author{Makoto {\sc Masumoto}}
\address{Department of Mathematics, Yamaguchi University, 753-8512, Japan}
\email{masumoto@yamaguchi-u.ac.jp}
\thanks{This research is supported in part by JSPS KAKENHI Grant Numbers 
26400140 and 15K04930.}
\date{}

\begin{document}
\begin{abstract}
In our previous work \cite{MasumotoPreprint}, for a given Riemann surface 
$Y_{0}$ with marked handle, we investigated geometric properties of the set of 
marked once-holed tori $X$ allowing holomorphic mappings of $X$ into $Y_{0}$. 
It turned out that it is a closed domain with Lipschitz boundary. In the 
present paper we show that the boundary is never smooth. Also, we evaluate the 
critical extremal length for the existence of holomorphic mappings in terms of 
hyperbolic lengths. 
\end{abstract}
\maketitle

\section{Introduction}
\label{sec:introductinon}

Let $R_{1}$ and $R_{2}$ be Riemann surfaces. It is a natural question whether 
there are holomorphic or conformal mappings of $R_{1}$ into $R_{2}$ with some 
geometric or analytic properties. In the present article we consider the 
problem in the case where $R_{1}$ is a once-holed torus, and look for 
handle-preserving mappings. 

Since Riemann surfaces of genus zero are conformally equivalent to plane 
regions, Riemann surfaces of positive genus should play the leading character 
in Riemann surfaces theory. Once-holed tori are topologically the simplest 
among the nonplanar Riemann surfaces. They are building blocks of Riemann 
surfaces of positive genus; every Riemann surface of positive genus $g$ is 
obtained from $g$ once-holed tori by suitable identification. Open disks are 
one of the simplest plane domains, and studies of functions on open disks are 
of fundamental importance for local theory. Thus studies of holomorphic 
mappings on once-holed tori would be significant for ``local theory" of 
holomorphic mappings between Riemann surfaces. While open disk are conformally 
equivalent to one another, once-holed tori are not. Hence we need to know which 
once-holed tori are included in a Riemann surface under consideration. This 
amounts to ask the existence of conformal mappings of once-holed tori into 
Riemann surfaces. 

For the existence of conformal mappings of once-holed tori several results are 
known. In \cite{Shiba1987} and \cite{Shiba1993} Shiba investigated the set of 
tori into which a given open Riemann surface of genus one is conformally 
embedded. His results give solutions to our problem in the case where $R_{2}$ 
is a torus. Also, we gave a characterization for the existence of conformal 
mappings of a once-holed torus into another explicitly in terms of finitely 
many extremal lengths (see \cite{Masumoto1995}). In \cite{Masumoto2007} we 
examined the set of once-holed tori that can be conformally embedded into a 
given Riemann surface of positive genus. For topologically finite surfaces 
Kahn-Pilgrim-Thurston \cite{KPT} has recently given a characterization for the 
existence of conformal embeddings in terms of extremal lengths. 

On the other hand, few results are known for the existence of holomorphic 
mappings of once-holed tori. If $R_{2}$ is a torus, then the Behnke-Stein 
theorem yields that any once-holed torus allows handle-preserving holomorphic 
mappings into $R_{2}$. If $R_{2}$ is not a torus, then it carries a hyperbolic 
metric. Since holomorphic mappings decrease hyperbolic lengths, we obtain 
necessary conditions for the existence of holomorphic mappings. However, they 
are not sufficient by a recent result of Bourque \cite{Bourque2016}. 

The space $\mathfrak{T}$ of marked once-holed tori is a three-dimensional 
real-analytic manifold with boundary. In our previous work 
\cite{MasumotoPreprint}, for a given Riemann surface $Y_{0}$ with marked 
handle, we investigated the set $\mathfrak{T}_{a}[Y_{0}]$ of marked once-holed 
tori $X$ such that there is a holomorphic mapping of $X$ into $Y_{0}$. We 
introduced a new condition called a handle condition to obtain the following 
two results. 

\begin{prop}[\mbox{\rm \cite[Theorem~1]{MasumotoPreprint}}]
\label{prop:holomorphic:Lipschitz}
$\mathfrak{T}_{a}[Y_{0}]$ is a closed domain with Lipschitz boundary, and is a 
retract of the whole space\/ $\mathfrak{T}$. 
\end{prop}

The second result is expressed in terms of a specific coordinate system on 
$\mathfrak{T}$. Every once-holed torus is realized as a slit torus. For 
$(\tau,s) \in \mathbb{H} \times [0,1)$ let $X_{\tau}^{(s)}$ denote the marked 
once-holed torus obtained from the marked torus $X_{\tau}$ of modulus $\tau$ by 
deleting a horizontal segment of length $s$, where $\mathbb{H}$ is the upper 
half-plane. 

\begin{prop}[\mbox{\rm \cite[Theorem~2]{MasumotoPreprint}}]
\label{prop:holomorphic:critical_extremal}
There exists a nonnegative number $\lambda_{a}[Y_{0}]$ such that 
\begin{list}
{{\rm (\roman{claim})}}{\usecounter{claim}
\setlength{\topsep}{0pt}
\setlength{\itemsep}{0pt}
\setlength{\parsep}{0pt}
\setlength{\labelwidth}{\leftmargin}}
\item if\/ $\im\tau \geqq 1/\lambda_{a}[Y_{0}]$, then there are no holomorphic 
mappings of $X_{\tau}^{(s)}$ into $Y_{0}$ for any $s$, while 

\item if\/ $\im\tau<1/\lambda_{a}[Y_{0}]$, then there are holomorphic mappings 
of $X_{\tau}^{(s)}$ into $Y_{0}$ for some $s$. 
\end{list}
\end{prop}

Propositions~\ref{prop:holomorphic:Lipschitz} 
and~\ref{prop:holomorphic:critical_extremal} raise the following natural 
questions: 

\begin{list}
{{\rm (\arabic{claim})}}{\usecounter{claim}
\setlength{\itemsep}{0pt}
\setlength{\parsep}{0pt}
\setlength{\labelwidth}{\leftmargin}}
\item Is the boundary of $\mathfrak{T}_{a}[Y_{0}]$ smooth? 

\item What is the value of $\lambda_{a}[Y_{0}]$? 
\end{list}

In the present paper we answer these questions. We first show that the boundary 
of $\mathfrak{T}_{a}[Y_{0}]$ is not smooth in most cases: 

\begin{thm}
\label{thm:holomorphic:nonsmooth}
If $Y_{0}$ is not a marked torus or a marked once-punctured torus, then the 
boundary of\/ $\mathfrak{T}_{a}[Y_{0}]$ is not smooth. 
\end{thm}

We prove Theorem~\ref{thm:holomorphic:nonsmooth} in \S3 after summarizing 
results of \cite{MasumotoPreprint} in \S2. In \S4 we compare 
$\mathfrak{T}_{a}[Y_{0}]$ with the set $\mathfrak{T}_{\sigma}[Y_{0}]$ of marked 
once-holed tori having longer geodesics than corresponding geodesics on 
$Y_{0}$. In the final section we give an answer to second question~(2) (see 
Theorem~\ref{thm:critical_extremal:holomorphic}). 

\section{Preliminaries}
\label{sec:preliminaries}

Let $R$ be a Riemann surface of positive genus; it may be compact or of 
infinite genus. It has one or more handles. A handle of $R$ is specified by a 
couple of loops on $R$. With this in mind we make the following definitions. A 
{\em mark of handle\/} of $R$ is, by definition, an ordered pair $\chi=\{a,b\}$ 
of simple loops $a$ and $b$ on $R$ whose geometric intersection number 
$a \times b$ is equal to one. A {\em Riemann surface with marked handle\/} 
means a pair $(R,\chi)$, where $R$ is a Riemann surface of positive genus and 
$\chi$ is a mark of handle of $R$. 

Let $Y_{1}:=(R_{1},\chi_{1})$ and $Y_{2}:=(R_{2},\chi_{2})$ be Riemann surfaces 
with marked handle, where $\chi_{j}=\{a_{j},b_{j}\}$ for $j=1,2$. If a 
continuous mapping $f:R_{1} \to R_{2}$ maps $a_{1}$ and $b_{1}$ onto loops 
freely homotopic to $a_{2}$ and $b_{2}$ on $R_{2}$, respectively, then we say 
that $f$ is a continuous mapping of $Y_{1}$ into $Y_{2}$ and use the notation 
$f:Y_{1} \to Y_{2}$. If $f:R_{1} \to R_{2}$ possesses some additional 
properties, then $f:Y_{1} \to Y_{2}$ is said to possess the same properties. 
For example, if $f:R_{1} \to R_{2}$ is conformal, then $f:Y_{1} \to Y_{2}$ is 
called conformal. Here, by a conformal mapping we mean a holomorphic injection; 
we do not require conformal mappings to be surjective. We consider continuous 
mappings of $Y_{1}$ into $Y_{2}$ preserve the handles specified by $\chi_{1}$ 
and $\chi_{2}$. 

A {\em once-holed torus\/} is, by definition, a noncompact Riemann surface of 
genus one with exactly one boundary component in the sense of 
Ker\'{e}kj\'{a}rt\'{o}-Sto\"{\i}low. For example, the Riemann surface obtained 
from a torus, that is, a compact Riemann surface of genus one, by removing one 
point is a once-holed torus, which will be referred to as a {\em once-punctured 
torus}. Note that once-holed tori are not bordered surfaces. A once-holed torus 
with marked handle is usually called a {\em marked once-holed torus}. The 
meaning of a {\em marked once-punctured torus\/} is obvious. 

Let $\mathfrak{T}$ denote the set of marked once-holed tori, where two marked 
once-holed tori are identified if there is a conformal mapping of one 
{\em onto\/} the other. As a set, it is the disjoint union of the 
Teichm\"{u}ller space $\mathfrak{T}'$ of a once-punctured torus and the reduced 
Teichm\"{u}ller space $\mathfrak{T}''$ of a once-holed torus that is not a 
once-punctured torus. 

There is a canonical injection $\Lambda$ of $\mathfrak{T}$ into 
$\mathbb{R}_{+}^{3}$; if $X=(T,\chi)$ with $\chi=\{a,b\}$, then $\Lambda(X)$ is 
the triplet of the extremal lengths of the free homotopy classes of $a$, $b$ 
and $ab^{-1}$. We know that 
$$
\mathfrak{L}:=\Lambda(\mathfrak{T})=\{\boldsymbol{x} \in \mathbb{R}_{+}^{3} \mid
Q(\boldsymbol{x})+4 \leqq 0\},
$$
where $Q(\boldsymbol{x})=Q(x_{1},x_{2},x_{3})=
x_{1}^{2}+x_{2}^{2}+x_{3}^{3}-2(x_{1}x_{2}+x_{2}x_{3}+x_{3}x_{1})$. Note that 
$\Lambda$ maps $\mathfrak{T}'$ and $\mathfrak{T}''$ onto the boundary 
$\partial\mathfrak{L}$ and the interior 
$\mathfrak{L}^{\circ}:=\mathfrak{L} \setminus \partial\mathfrak{L}$, 
respectively. Moreover, the restrictions 
$\Lambda|_{\mathfrak{T}'}:\mathfrak{T}' \to \partial\mathfrak{L}$ and 
$\Lambda|_{\mathfrak{T}''}:\mathfrak{T}'' \to \mathfrak{L}^{\circ}$ are 
real-analytic diffeomorphisms. We regard $\mathfrak{T}$ as a three-dimensional 
real-analytic manifold with boundary so that 
$\Lambda:\mathfrak{T} \to \mathfrak{L}$ is a real-analytic diffeomorphism. In 
the rest of the article we use the notations $\partial\mathfrak{T}$ and 
$\mathfrak{T}^{\circ}$ instead of $\mathfrak{T}'$ and $\mathfrak{T}''$, 
respectively. For details, see \cite[\S7]{Masumoto1995}. 

Now, fix a Riemann surface $Y_{0}=(R_{0},\chi_{0})$ with marked handle. For a 
given marked once-holed torus $X$ there may or may not exist holomorphic 
mappings of $X$ into $Y_{0}$. We are interested in the set of marked once-holed 
tori which allow holomorphic mappings into $Y_{0}$. We denote by 
$\mathfrak{T}_{a}[Y_{0}]$ (resp.\ $\mathfrak{T}_{c}[Y_{0}]$) the set of 
$X \in \mathfrak{T}$ such that there exists a holomorphic (resp.\ conformal) 
mapping of $X$ into $Y_{0}$. In our previous work \cite{MasumotoPreprint} we 
introduced handle conditions to investigated geometric properties of 
$\mathfrak{T}_{a}[Y_{0}]$ and $\mathfrak{T}_{c}[Y_{0}]$. We recall the 
definition of a handle condition. 

For $X,X' \in \mathfrak{T}$ we say that $X$ is smaller than $X'$ and write 
$X \preceq X'$ if there is a conformal mapping of $X$ into $X'$. The relation 
$\preceq$ is then an order relation on $\mathfrak{T}$. 

A mathematical statement $\mathcal{P}(X)$, where the free variable $X$ ranges 
over $\mathfrak{T}$, is called a {\em handle condition\/} if 
$\mathcal{P}(X_{1})$ implies $\mathcal{P}(X_{2})$ for all 
$X_{1},X_{2} \in \mathfrak{T}$ with $X_{2} \preceq X_{1}$. Important examples 
are the statements ``there is a holomorphic mapping of $X$ into $Y_{0}$" and 
``there is a conformal mapping of $X$ into $Y_{0}$," which will be denoted by 
$\mathcal{P}_{a}(X)$ and $\mathcal{P}_{c}(X)$, respectively. For 
$\nu \in \mathbb{N}$ the statement 
$$
\mathcal{P}_{\nu}(X)=
\text{``There is a holomorphic mapping $f:X \to Y_{0}$ with $d(f)<\nu+1$"}
$$
is another handle condition, where $d(f)$ is the supremum of the cardinal 
numbers of $f^{-1}(p)$, $p \in R_{0}$. Note that 
$\mathcal{P}_{1}(X)=\mathcal{P}_{c}(X)$. 

Set $\mathfrak{T}[\mathcal{P}]=\{X \in \mathfrak{T} \mid \mathcal{P}(X)\}$. 
Then we have the following proposition. 

\begin{prop}[\mbox{\rm \cite[Theorem~3]{MasumotoPreprint}}]
\label{prop:handle_condition:Lipschitz}
If\/ $\mathfrak{T}[\mathcal{P}] \neq \varnothing$, then its interior\/ 
$\mathfrak{T}^{\circ}[\mathcal{P}]$ is a domain with Lipschitz boundary. 
\end{prop}

\begin{rem}
In the case where $\mathfrak{T}[\mathcal{P}]=\mathfrak{T}$, we consider 
$\mathfrak{T}[\mathcal{P}]$ to have a Lipschitz boundary though the boundary 
$\partial\mathfrak{T}[\mathcal{P}]$ is in fact an empty set. 
\end{rem}

Actually, we can show more. The eigenvalues of the coefficient matrix of the 
quadratic form $Q(\boldsymbol{x})$ are $-1$ and $2$. The corresponding 
eigenspaces $V_{-1}$ and $V_{2}$ re, respectively, the line $x_{1}=x_{2}=x_{3}$ 
and the plane $x_{1}+x_{2}+x_{3}=0$. Let $\boldsymbol{e}=
\bigl(1/\sqrt{3\,},1/\sqrt{3\,},1/\sqrt{3\,}\bigr) \in V_{-1}$. 

\begin{prop}[\mbox{\cite[Proposition~4]{MasumotoPreprint}}]
\label{prop:handle_condition:graph}
There is a Lipschitz continuous function $e[\mathcal{P}](\,\cdot\,)$ on $V_{2}$ 
such that 
$$
\Lambda(\mathfrak{T}^{\circ}[\mathcal{P}])=
\{\boldsymbol{\zeta}+t\boldsymbol{e} \mid
\boldsymbol{\zeta} \in V_{2}[\mathcal{P}],
t>e[\mathcal{P}](\boldsymbol{\zeta})\},
$$
provided that $\mathfrak{T}[\mathcal{P}] \neq \varnothing$. 
\end{prop}

Since $\mathfrak{T}_{a}[Y_{0}]=\mathfrak{T}[\mathcal{P}_{a}]$, 
Proposition~\ref{prop:holomorphic:Lipschitz} follows from 
Proposition~\ref{prop:handle_condition:Lipschitz} together with the fact that 
$\mathfrak{T}_{a}[Y_{0}]$ is closed. The set 
$\mathfrak{T}_{c}[Y_{0}]:=\mathfrak{T}[\mathcal{P}_{c}]$ possesses the same 
properties. In fact, setting 
$\mathfrak{T}_{\nu}[Y_{0}]=\mathfrak{T}[\mathcal{P}_{\nu}]$, we deduce the 
following proposition. 

\begin{prop}[\mbox{\rm \cite[Corollary~1]{MasumotoPreprint}}]
\label{prop:finite_degree:Lipschitz}
The sets\/ $\mathfrak{T}_{\nu}[Y_{0}]$, $\nu \in \mathbb{N}$, are closed 
domains with Lipschitz boundary, and are retracts of\/ $\mathfrak{T}$. 
\end{prop}

Every marked once-holed torus is realized as a horizontal slit torus (see 
Shiba \cite{Shiba1987}). Specifically, for each point $\tau$ in the upper 
half-plane $\mathbb{H}$, let $G_{\tau}$ be the additive subgroup of 
$\mathbb{C}$ generated by $1$ and $\tau$. Then $T_{\tau}:=\mathbb{C}/G_{\tau}$ 
is a torus. Let $\pi_{\tau}:\mathbb{C} \to T_{\tau}$ be the natural projection, 
and set $a_{\tau}=\pi_{\tau}([0,1])$ and $b_{\tau}=\pi_{\tau}([0,\tau])$, where 
$[z_{1},z_{2}]$ stands for the oriented line segment joining $z_{1}$ with 
$z_{2}$; if $z_{1}=z_{2}$, then $[z_{1},z_{2}]$ denotes the singleton 
$\{z_{1}\}$. Then $\chi_{\tau}:=\{a_{\tau},b_{\tau}\}$ is a mark of handle of 
$T_{\tau}$, and we obtain a marked torus $X_{\tau}:=(T_{\tau},\chi_{\tau})$. 

Now, for $s \in [0,1)$ set 
$T_{\tau}^{(s)}=T_{\tau} \setminus \pi_{\tau}([0,s])$; it is a once-holed 
torus. We choose a mark 
$\chi_{\tau}^{(s)}=\bigl\{a_{\tau}^{(s)},b_{\tau}^{(s)}\bigr\}$ of handle of 
$T_{\tau}^{(s)}$ so that the inclusion mapping of $T_{\tau}^{(s)}$ into 
$T_{\tau}$ is a conformal mapping of 
$X_{\tau}^{(s)}:=\bigl(T_{\tau}^{(s)},\chi_{\tau}^{(s)}\bigr)$ into $X_{\tau}$. 
Then the correspondence $(\tau,s) \mapsto X_{\tau}^{(s)}$ is a homeomorphism of 
$\mathbb{H} \times [0,1)$ onto $\mathfrak{T}$, which is a real-analytic 
diffeomorphism on $\mathbb{H} \times (0,1)$ (see \cite[\S4]{MasumotoPreprint}). 
In other words, its inverse $\Sigma:X_{\tau}^{(s)} \mapsto (\tau,s)$ serves as 
a global topological coordinate system on $\mathfrak{T}$. 

Let $\Pi:\mathbb{H} \times [0,1) \to \mathbb{H}$ be the natural projection. 
Then for any handle condition $\mathcal{P}(X)$ the image 
$\mathbb{H}[\mathcal{P}]:=\Pi \circ \Sigma(\mathfrak{T}[\mathcal{P}])$ is a 
horizontal strip. To be more precise for 
$t \in \bar{\mathbb{R}}_{+}:=\mathbb{R}_{+} \cup \{+\infty\}=[0,+\infty]$ set 
$\mathcal{H}(t)=\{t \in \mathbb{C} \mid 0<\im z<t\}$ and let 
$\bar{\mathcal{H}}(t)$ denote its closure in $\mathbb{H}$. Note that 
$\mathcal{H}(0)=\bar{\mathcal{H}}(0)=\varnothing$ and 
$\mathcal{H}(+\infty)=\bar{\mathcal{H}}(+\infty)=\mathbb{H}$. 

\begin{prop}[\mbox{\rm \cite[Theorem~4]{MasumotoPreprint}}]
\label{prop:handle_condition:strip}
For every handle condition $\mathcal{P}(X)$ there exists a constant 
$\lambda[\mathcal{P}] \in \bar{\mathbb{R}}_{+}$ such that 
$$
\mathcal{H}\biggl(\frac{1}{\,\lambda[\mathcal{P}]\,}\biggr) \subset
\mathbb{H}[\mathcal{P}] \subset
\bar{\mathcal{H}}\biggl(\frac{1}{\,\lambda[\mathcal{P}]\,}\biggr),
$$
where $1/0=+\infty$ and $1/(+\infty)=0$. 
\end{prop}

We set $\lambda_{a}[Y_{0}]=\lambda[\mathcal{P}_{a}]$ and 
$\lambda_{c}[Y_{0}]=\lambda[\mathcal{P}_{c}]$. They are referred to as the 
{\em critical extremal lengths\/} for the existence of holomorphic and 
conformal mappings of marked once-holed tori into $Y_{0}$, respectively. Most 
part of Proposition~\ref{prop:holomorphic:critical_extremal} follows from 
Proposition~\ref{prop:handle_condition:strip}. Note, however, that 
Proposition~\ref{prop:holomorphic:critical_extremal} asserts that the identity 
$\mathbb{H}[\mathcal{P}_{a}]=\mathcal{H}(1/\lambda[\mathcal{P}_{a}])$ actually 
holds. On the other hand, in general,$ \mathcal{H}(1/\lambda[\mathcal{P}_{c}])$ 
is a proper subset of $\mathbb{H}[\mathcal{P}_{c}]$ (see 
\cite[Example~13]{MasumotoPreprint}). 

\section{Non-smoothness of boundaries}
\label{sec:nonsmoothness}

Propositions~\ref{prop:holomorphic:Lipschitz} 
and~\ref{prop:finite_degree:Lipschitz} show that $\mathfrak{T}_{a}[Y_{0}]$ and 
$\mathfrak{T}_{\nu}[Y_{0}]$ have Lipschitz boundaries. It is then natural to 
ask whether the boundaries are smooth or not. As for 
$\mathfrak{T}_{c}[Y_{0}]=\mathfrak{T}_{1}[Y_{0}]$ we know that the answer is 
negative in general. In fact, if $Y_{0}$ is a marked once-holed torus, then 
$\Lambda(\mathfrak{T}_{c}[Y_{0}])$ is a cone with vertex at $\Lambda(Y_{0})$ 
and hence the boundary of $\mathfrak{T}_{c}[Y_{0}]$ is not smooth at $Y_{0}$ 
(see \cite[Example~10]{MasumotoPreprint}). Our first result, 
Theorem~\ref{thm:holomorphic:nonsmooth}, claims that the boundary of 
$\mathfrak{T}_{a}[Y_{0}]$ is not, either, for most cases. 

For the proof of Theorem~\ref{thm:holomorphic:nonsmooth} we define the handle 
covering surface of a Riemann surface $Y_{0}$ with marked handle. There is a 
Riemann surface $\tilde{Y}_{0}=(\tilde{R}_{0},\tilde{\chi}_{0})$ with marked 
handle together with a holomorphic mapping $\pi_{0}:\tilde{Y}_{0} \to Y_{0}$ 
such that 
\begin{list}
{{\rm (\roman{claim})}}{\usecounter{claim}
\setlength{\itemsep}{0pt}
\setlength{\parsep}{0pt}
\setlength{\labelwidth}{\leftmargin}}
\item the fundamental group of $\tilde{R}_{0}$ is generated by the loops in 
$\tilde{\chi}_{0}$, and 

\item $\pi_{0}:\tilde{R}_{0} \to R_{0}$ is a covering map. 
\end{list}
We call $\tilde{Y}_{0}$ the {\em handle covering surface\/} of $Y_{0}$. The 
following lemma is easily verified. 

\begin{lem}
\label{lem:handle_covering}
Let $Y_{0}$ be a Riemann surface of marked handle and $\tilde{Y}_{0}$ its 
handle covering surface. 
\begin{list}
{{\rm (\roman{claim})}}{\usecounter{claim}
\setlength{\topsep}{0pt}
\setlength{\itemsep}{0pt}
\setlength{\parsep}{0pt}
\setlength{\labelwidth}{\leftmargin}}
\item If $Y_{0}$ is not a marked torus, then $\tilde{Y}_{0}$ is a marked 
once-holed torus. If $Y_{0}$ is a marked torus or a marked once-holed torus, 
then $\tilde{Y}_{0}=Y_{0}$. 

\item If $\tilde{Y}_{0}$ is a marked once-punctured torus, then so is $Y_{0}$. 

\item $\tilde{Y}_{0} \in \mathfrak{T}_{a}[Y_{0}]$. 

\item $\mathfrak{T}_{a}[\tilde{Y}_{0}]=\mathfrak{T}_{a}[Y_{0}]$. 
\end{list}
\end{lem}

\begin{proof}[Proof of Theorem~\ref{thm:holomorphic:nonsmooth}]
By Lemma~\ref{lem:handle_covering}~(iv) we may assume from the outset that 
$Y_{0}$ is an element of $\mathfrak{T}^{\circ}$. We employ Fenchel-Nielsen 
coordinates (see Buser \cite{Buser1992}). To be specific, let 
$X=(T,\chi) \in \mathfrak{T}$, where $\chi=\{a,b\}$. The once-holed torus $T$ 
carries a hyperbolic metric, whose curvature is normalized to be $-1$. Denote 
by $l(X)$ the length of the hyperbolic geodesic $\alpha$ freely homotopic to 
$a$. Let $\theta(X)$ stand for the twist parameter along $\alpha$. Also, let 
$l'(X)$ be the infimum of hyperbolic lengths of loops freely homotopic to 
$aba^{-1}b^{-1}$. Clearly, $l'(X)$ vanishes if and only if $X$ is a marked 
once-punctured torus. Setting $\Phi(X)=\bigl(l(X),l'(X),\theta(X)\bigr)$, we 
obtain a homeomorphism of $\mathfrak{T}$ onto 
$(0,+\infty) \times [0,+\infty) \times \mathbb{R}$, which is a real-analytic 
diffeomorphism of $\mathfrak{T}^{\circ}$ onto 
$(0,+\infty) \times (0,+\infty) \times \mathbb{R}$. 

If $X \in \mathfrak{T}_{a}[Y_{0}]$, then $l(X) \geqq l(Y_{0})$ and 
$l'(X) \geqq l'(Y_{0})$ since holomorphic mappings decrease hyperbolic metrics. 
As $Y_{0} \in \mathfrak{T}_{a}[Y_{0}]$, these inequalities imply that $Y_{0}$ 
lies on the boundary $\partial\mathfrak{T}_{a}[Y_{0}]$ and that 
$\partial\mathfrak{T}_{a}[Y_{0}]$ is not smooth at $Y_{0}$, provided that 
$Y_{0} \not\in \partial\mathfrak{T}$. Theorem~\ref{thm:holomorphic:nonsmooth} 
has been thus established. 
\end{proof}

\begin{rem}
If $Y_{0}$ is a marked torus, then $\mathfrak{T}_{a}[Y_{0}]$ coincides with the 
whole space $\mathfrak{T}$ (see \cite[Example~9]{MasumotoPreprint}). Thus its 
boundary is an empty set. For marked once-punctured tori $Y_{0}$ the boundary 
of $\Phi(\mathfrak{T}_{a}[Y_{0}])$ is not smooth. However, we do not know 
whether the boundary of $\Lambda(\mathfrak{T}_{a}[Y_{0}])$ is smooth or not. 
\end{rem}

It is obvious that 
$$
\mathfrak{T}_{c}[Y_{0}]=\mathfrak{T}_{1}[Y_{0}] \subset \cdots \subset
\mathfrak{T}_{\nu}[Y_{0}] \subset \mathfrak{T}_{\nu+1}[Y_{0}] \subset \cdots
\subset \mathfrak{T}_{a}[Y_{0}].
$$
If $Y_{0} \in \mathfrak{T}^{\circ}$, then both of 
$\partial\mathfrak{T}_{c}[Y_{0}]$ and $\partial\mathfrak{T}_{a}[Y_{0}]$ 
contains $Y_{0}$ and are not smooth at $Y_{0}$. We thus have the following 
corollary to Theorem~\ref{thm:holomorphic:nonsmooth}. 

\begin{cor}
\label{cor:nonsmooth}
If $Y_{0}$ is a marked once-holed torus which is not a marked once-punctured 
torus, then for any positive integer $\nu$ the boundary of\/ 
$\mathfrak{T}_{\nu}[Y_{0}]$ is not smooth. 
\end{cor}

\section{Hyperbolic length spectra}
\label{sec:length_spectra}

Let $W$ be a free group generated by two elements. We regard it as the set of 
reduced words $w(u,v)$ of two letters $u$ and $v$. The unit is the void word. 
We denote by $W^{*}$ the subset of non-unit elements. 

In general, let $Y=(R,\chi)$, where $\chi=\{a,b\}$, be a Riemann surface with 
marked handle. For $w=w(u,v) \in W^{*}$ the notation $w(a,b)$ denotes a loop on 
$R$. We set $w(Y)=w(a,b)$. In particular, $u(Y)=a$ and $v(Y)=b$. Let $l(Y,w)$ 
be the infimum of the hyperbolic lengths of loops in the free homotopy class 
$\Gamma(Y,w)$ of $w(Y)$ on $R$, provided that $R$ is not a torus. In the case 
where $R$ is a torus, we set $l(Y,w)=0$ for convenience. 

Now, fix a Riemann surface $Y_{0}$ with marked handle, and let 
$\mathfrak{T}_{\sigma}[Y_{0}]$ be the set of $X \in \mathfrak{T}$ for which 
$l(X,w) \geqq l(Y_{0},w)$ for all $w \in W^{*}$. Since holomorphic mappings 
decrease hyperbolic lengths, $\mathfrak{T}_{a}[Y_{0}]$ is included in 
$\mathfrak{T}_{\sigma}[Y_{0}]$. It follows from Bourque \cite{Bourque2016} that 
$\mathfrak{T}_{a}[Y_{0}]$ is in general a proper subset of 
$\mathfrak{T}_{\sigma}[Y_{0}]$. The following theorem claims more: 

\begin{thm}
\label{thm:length_spectra}
Let $Y_{0}$ be a Riemann surface with marked handle which is not a marked 
torus. Then 
\begin{list}
{{\rm (\roman{claim})}}{\usecounter{claim}
\setlength{\topsep}{0pt}
\setlength{\itemsep}{0pt}
\setlength{\parsep}{0pt}
\setlength{\labelwidth}{\leftmargin}}
\item $\mathfrak{T}_{\sigma}[Y_{0}]$ is a closed domain with Lipschitz 
boundary, 

\item its boundary $\partial\mathfrak{T}_{\sigma}[Y_{0}]$ meets 
$\mathfrak{T}_{a}[Y_{0}]$ exactly  at one point, and 

\item $\mathfrak{T}_{\sigma}[Y_{0}] \setminus \mathfrak{T}_{a}[Y_{0}]$ is 
homeomorphic to $\mathbb{C}^{*} \times [0,1)$, where 
$\mathbb{C}^{*}=\mathbb{C} \setminus \{0\}$. 
\end{list}
\end{thm}

\begin{rem}
If $Y_{0}$ is a marked torus, then 
$\mathfrak{T}_{\sigma}[Y_{0}]=\mathfrak{T}_{a}[Y_{0}]=\mathfrak{T}$ (see 
\cite[Example~12]{MasumotoPreprint}). 
\end{rem}

For the proof of Theorem~\ref{thm:length_spectra} we introduce some notations, 
and prepare several lemmas. We first remark the following lemma. 

\begin{lem}
\label{lem:handle_covering:length_spectra}
If $\tilde{Y}_{0}$ is the handle covering surface of $Y_{0}$, then 
$\mathfrak{T}_{\sigma}[\tilde{Y}_{0}]=\mathfrak{T}_{\sigma}[Y_{0}]$. 
\end{lem}

Let $D$ be a doubly connected Riemann surface. Denote by $\lambda(D)$ the 
extremal length of the free homotopy class $\Gamma(D)$ of a simple loop 
separating the boundary components of $D$. Unless $D$ is conformally equivalent 
to the punctured plane $\mathbb{C} \setminus \{0\}$, it carries a hyperbolic 
metric. Let $l(D)$ stand for the infimum of the hyperbolic lengths of loops in 
$\Gamma(D)$. Define $l(D)=0$ if $D$ is conformally equivalent to 
$\mathbb{C} \setminus \{0\}=0$. Note that the identity 
\begin{equation}
\label{eq:hyperbolic-extremal:annulus}
\lambda(D)=\frac{1}{\,\pi\,}l(D)
\end{equation}
holds. 

\begin{rem}
If $a$ is a simple loop on $D$ separating the boundary components of $D$, then 
so is $a^{-1}$. Though the free homotopy classes $\Gamma^{+}$ and $\Gamma^{-}$ 
of $a$ and $a^{-1}$, respectively, are disjoint, their extremal lengths are 
identical with each other. The common value is denoted by $\lambda(D)$. In the 
sequel $\Gamma(D)$ will represent one of $\Gamma^{+}$ and $\Gamma^{-}$. 
\end{rem}

Let $Y=(R,\chi)$ be a Riemann surface with marked handle, and take 
$w \in W^{*}$. Let $D$ be the annular covering surface of $R$ with respect to 
the loop $w(Y)$ (see \cite[\S3]{MRR}). Thus $D$ is a doubly connected Riemann 
surface, and there is a holomorphic covering map $\pi$ of $D$ onto $R$ which 
maps $\Gamma(D)$ into $\Gamma(Y,w)$. Clearly, we have 
\begin{equation}
\label{eq:annular_covering:hyperbolic}
l(D)=l(Y,w).
\end{equation}

\begin{lem}[\mbox{\rm Wolpert \cite[Lemma~3.1]{Wolpert1979}}]
\label{lem:length_spectra:qc}
Let $Y_{1}$ and $Y_{2}$ be Riemann surfaces with marked handle. If there is 
a $K$-quasiconformal mapping of $Y_{1}$ onto $Y_{2}$, then 
$$
\frac{1}{\,K\,}l(Y_{1},w) \leqq l(Y_{2},w) \leqq Kl(Y_{1},w)
$$
for all $w \in W^{*}$. 
\end{lem}

In fact, let $D_{j}$ be the annular covering surface of $Y_{j}$ with respect 
to $w(Y_{j})$. If there is a $K$-quasiconformal mapping of $Y_{1}$ onto 
$Y_{2}$, then it is lifted to a $K$-quasiconformal mapping $D_{1}$ onto $D_{2}$ 
which maps $\Gamma(D_{1})$ to $\Gamma(D_{2})$. Since extremal lengths are 
quasi-invariant under quasiconformal mappings, the lemma is an immediate 
consequence of~\eqref{eq:hyperbolic-extremal:annulus} 
and~\eqref{eq:annular_covering:hyperbolic}. 

\begin{lem}
\label{lem:length_spectra:continuity}
For each $w \in W^{*}$ the function $l(\,\cdot\,,w)$ is continuous on 
$\mathfrak{T}$. 
\end{lem}

\begin{proof}
It follows from Lemma~\ref{lem:length_spectra:qc} that $l(\,\cdot\,,w)$ is 
continuous on $\mathfrak{T}^{\circ}$. To show that it is also continuous at 
each point of $\partial\mathfrak{T}$, take an arbitrary marked once-punctured 
torus $X_{\tau_{0}}^{(0)}$. For $s \in [0,1)$ the annular covering surface of 
$X_{\tau_{0}}^{(s)}$ with respect to the loop $w\bigl(X_{\tau_{0}}^{(s)}\bigr)$ 
is conformally equivalent to the annulus $A(s):=\bigl\{z \in \mathbb{C} \mid
\exp\bigl(-2\pi^{2}/l(X_{\tau_{0}}^{(s)},w)\bigr)<|z|<1\bigr\}$. 
The inclusion mapping $\iota_{s}$ of $X_{\tau_{0}}^{(s)}$ into 
$X_{\tau_{0}}^{(0)}$ induces a conformal mapping of $A(s)$ into $A(0)$. Observe 
that the function $s \mapsto l(X_{\tau_{0}}^{(s)},w)$ is increasing. Since 
$\iota_{s}$ tends to the identity mapping of $X_{\tau_{0}}^{(0)}$ onto itself 
as $s \to 0$, we see that $l(X_{\tau_{0}}^{(s)},w)$ converges to 
$l(X_{\tau_{0}}^{(0)},w)$. 

Now, for $(\tau,s) \in \mathbb{H} \times [0,1)$ the $\mathbb{R}$-linear mapping 
$F_{\tau}$ of $\mathbb{C}$ onto itself with $F_{\tau}(1)=1$ and 
$F_{\tau}(\tau_{0})=\tau$ induces a quasiconformal mapping of 
$X_{\tau_{0}}^{(s)}$ onto $X_{\tau}^{(s)}$ whose maximal dilatation is equal to 
$e^{d_{\mathbb{H}}(\tau_{0},\tau)}$, where $d_{\mathbb{H}}(\tau_{0},\tau)$ is 
the distance between $\tau_{0}$ and $\tau$ with respect to the hyperbolic 
metric on $\mathbb{H}$. We apply Lemma~\ref{lem:length_spectra:qc} to obtain 
\begin{align*}
 & \bigl|l\bigl(X_{\tau}^{(s)},w\bigr)-l\bigl(X_{\tau_{0}}^{(0)},w\bigr)\bigr|
\\ \leqq &
\bigl|l\bigl(X_{\tau}^{(s)},w\bigr)-l\bigl(X_{\tau_{0}}^{(s)},w\bigr)\bigr|+
\bigl|l\bigl(X_{\tau_{0}}^{(s)},w\bigr)-l\bigl(X_{\tau_{0}}^{(0)},w\bigr)\bigr|
\\
\leqq & 
\bigl(e^{d_{\mathbb{H}}(\tau_{0},\tau)}-e^{-d_{\mathbb{H}}(\tau_{0},\tau)}\bigr)
l\bigl(X_{\tau_{0}}^{(s)},w\bigr)+
\bigl|l\bigl(X_{\tau_{0}}^{(s)},w\bigr)-l\bigl(X_{\tau_{0}}^{(0)},w\bigr)\bigr|.
\end{align*}
Consequently, 
$l\bigl(X_{\tau}^{(s)},w\bigr) \to l\bigl(X_{\tau_{0}}^{(0)},w\bigr)$ as 
$(\tau,s) \to (\tau_{0},0)$, which means that the function $l(\,\cdot\,,w)$ is 
continuous at $X_{\tau_{0}}^{(0)}$, as desired. 
\end{proof}

\begin{cor}
\label{cor:length_spectra:continuity}
$\mathfrak{T}_{\sigma}[Y_{0}]$ is a closed subset of\/ $\mathfrak{T}$. 
\end{cor}

\begin{proof}[Proof of Theorem~\ref{thm:length_spectra}]
Since conformal mappings decrease hyperbolic metrics, the statement 
$$
\mathcal{P}_{\sigma}(X):=``\text{$l(X,w) \geqq l(Y_{0},w)$ for all 
$w \in W^{*}$}"
$$
is a handle condition  (see \cite[Example~8]{MasumotoPreprint}). As 
$\mathfrak{T}[P_{\sigma}]=\mathfrak{T}_{\sigma}[Y_{0}]$, 
Proposition~\ref{prop:handle_condition:Lipschitz} together with 
Corollary~\ref{cor:length_spectra:continuity} implies assertion~(i). 

In order to show assertion~(ii), by Lemmas~\ref{lem:handle_covering}~(iv) 
and~\ref{lem:handle_covering:length_spectra}, 
we have only to consider the case where $Y_{0}$ is a marked once-holed torus, 
say, $X_{\tau_{0}}^{(s_{0})}$. If there is a holomorphic mapping of a marked 
once-punctured torus $X_{\tau}^{(0)}$ into $Y_{0}=X_{\tau_{0}}^{(s_{0})}$, then 
it is extended to a holomorphic mapping between the marked tori $X_{\tau}$ and 
$X_{\tau_{0}}$, which must be conformal. Consequently, $Y_{0}$ is also a marked 
once-punctured torus identical with $X_{\tau}^{(0)}$. We have shown that 
$\mathfrak{T}_{a}[Y_{0}] \cap \partial\mathfrak{T}$ is empty or a singleton and 
that in the latter case $\mathfrak{T}_{a}[Y_{0}] \cap \partial\mathfrak{T}$ 
consists only of $Y_{0}$. 

Take an arbitrary $X \in \mathfrak{T}_{a}[Y_{0}] \setminus \{Y_{0}\}$. We 
employ arguments in \cite{Bourque2016} to prove that $X$ lies in the interior 
of $\mathfrak{T}_{\sigma}[Y_{0}]$. Let $f$ be a holomorphic mapping of $X$ into 
$Y_{0}$. Let $\rho_{X}=\rho_{X}(z)\,|dz|$ and 
$\rho_{Y_{0}}=\rho_{Y_{0}}(\zeta)\,|d\zeta|$ denote the hyperbolic metrics on 
$X$ and $Y_{0}$, respectively. By Schwarz's lemma the continuous function 
$(f^{*}\rho_{Y_{0}})/\rho_{X}$ is strictly less than one pointwise, where 
$f^{*}\rho_{Y_{0}}$ stands for the pull-back of $\rho_{Y_{0}}$ by $f$. The 
convex core $C$ of $X$ is compact and hence there is $c \in (0,1)$ for which 
$(f^{*}\rho_{Y_{0}})/\rho_{X}<c$ on $C$. For $w \in W^{*}$ let $\gamma_{w}$ be 
the closed geodesic on $X$ freely homotopic to $w(X)$. Since $\gamma_{w}$ lies 
in $C$, we have 
$$
l(Y,w) \leqq \int_{f_{*}\gamma_{w}} \rho_{Y_{0}}=
\int_{\gamma_{w}} f^{*}\rho_{Y_{0}} \leqq c\int_{\gamma_{w}} \rho_{X}=cl(X,w).
$$
There is a neighborhood $U$ of $X$ such that for any $X' \in U$ there is a 
$(1/c)$-quasiconformal mapping of $X$ onto $X'$. Applying 
Lemma~\ref{lem:length_spectra:qc}, we infer that 
$X' \in \mathfrak{T}_{\sigma}[Y_{0}]$. Thus 
$U \subset \mathfrak{T}_{\sigma}[Y_{0}]$, or, $X$ is an interior point of 
$\mathfrak{T}_{\sigma}[Y_{0}]$, as claimed. We have proved assertion~(ii). 

Assertion~(iii) is now an easy consequence of 
Proposition~\ref{prop:handle_condition:graph}. This completes the proof. 
\end{proof}

\begin{rem}
We see from the proof that the element in 
$\mathfrak{T}_{a}[Y_{0}] \cap \partial\mathfrak{T}_{\sigma}[Y_{0}]$ is the 
handle covering surface $\tilde{Y}_{0}$ of $Y_{0}$. There is exactly one 
holomorphic mapping of $\tilde{Y}_{0}$ into $Y_{0}$ (see \cite{Masumoto2013}). 
Therefore, if $Y_{0}$ is not a marked once-holed torus, then there are no 
holomorphic mappings $f:\tilde{Y}_{0} \to Y_{0}$ with $d(f)<\infty$. 
\end{rem}

\section{Critical extremal lengths}
\label{sec:critical_extremal}

The purpose of this section is to evaluate the critical extremal lengths for 
the existence of holomorphic and conformal mappings of marked once-holed tori 
into a Riemann surface with marked handle. Let $Y=(R,\chi)$, where 
$\chi=\{a,b\}$, be a Riemann surface with marked handle. Recall that $W$ is the 
free group generated by $u$ and $v$. Set $\Gamma(Y)=\Gamma(Y,u)$ and 
$l(Y)=l(Y,u)$. Thus $\Gamma(Y)$ is the free homotopy class of $a=u(Y)$. Let 
$\lambda(Y)$ stand for its extremal length. Note that 
$\lambda\bigl(X_{\tau}^{(s)}\bigr)=1/\im\tau$. 

Now, fix a Riemann surface $Y_{0}$ with marked handle. We begin with evaluating 
the critical extremal length $\lambda_{a}[Y_{0}]$ for the existence of 
holomorphic mappings. 

\begin{thm}
\label{thm:critical_extremal:holomorphic}
$\lambda_{a}[Y_{0}]=\dfrac{1}{\,\pi\,}l(Y_{0})$. 
\end{thm}

\begin{proof}
If $Y_{0}$ is a marked torus, then $\mathfrak{T}_{a}[Y_{0}]=\mathfrak{T}$ (see 
the remark following Theorem~\ref{thm:length_spectra}) and hence 
$\lambda_{a}[Y_{0}]=0$. Since $l(Y_{0})=0$ by definition, we see that the 
theorem is valid in this case. 

Next suppose that $Y_{0}$ is not a marked torus. 
By Lemma~\ref{lem:handle_covering}~(iv) we have only to consider the case where 
$Y_{0}$ is a marked once-holed torus. Let $X_{\tau}^{(s)}$ be an arbitrary 
element of $\mathfrak{T}_{a}[Y_{0}]$. The image $D_{\tau}$ of the horizontal 
strip $\{z \in \mathbb{C} \mid 0<\im z<\im \tau\}$ by the projection 
$\pi_{\tau}:\mathbb{C} \to T_{\tau}=\mathbb{C}/G_{\tau}$ is a doubly connected 
domain on $T_{\tau}$ and is included in $T_{\tau}^{(s)}$. Since 
$\lambda\bigl(X_{\tau}^{(s)}\bigr)=\lambda(D_{\tau})=1/\im\tau$, we see 
from~\eqref{eq:hyperbolic-extremal:annulus} that 
$$
\lambda\bigl(X_{\tau}^{(s)}\bigr)=\frac{1}{\,\pi\,}l(D_{\tau})>
\frac{1}{\,\pi\,}l\bigl(X_{\tau}^{(s)}\bigr)
$$
(cf.\ Maskit \cite[Proposition~1]{Maskit1985}). As holomorphic mappings 
decrease hyperbolic lengths, we obtain 
$\lambda\bigl(X_{\tau}^{(s)}\bigr)>(1/\pi)l(Y_{0})$, which implies 
$$
\lambda_{a}[Y_{0}] \geqq \frac{1}{\,\pi\,}l(Y_{0}).
$$

To show the opposite inequality we employ the annular covering surface $D_{0}$ 
of $R_{0}$ with respect to the loop $a_{0}$, where $Y_{0}=(R_{0},\chi_{0})$ and 
$\chi_{0}=\{a_{0},b_{0}\}$. For any $\varepsilon>0$ choose a doubly connected 
and relatively compact subdomain $D$ of $D_{0}$ with 
$\Gamma(D) \subset \Gamma(D_{0})$ so that 
$$
l(D)<l(D_{0})+\varepsilon=l(Y_{0})+\varepsilon.
$$
We further assume that the components of $\partial D$ are simple loops on 
$D_{0}$. Let $\pi_{0}:D_{0} \to R_{0}$ be the covering map. Since the closure 
$\overline{\pi_{0}(D)}$ is compact in $R_{0}$, we can find a simple loop on 
$R_{0}$ which is freely homotopic to $b_{0}$ and meets 
$R_{0} \setminus \overline{\pi_{0}(D)}$. Lifting the loop and deforming the 
lift, we obtain a simple arc $\tilde{b}$ on $D_{0}$ such that 
\begin{list}
{{\rm (\roman{claim})}}{\usecounter{claim}
\setlength{\topsep}{0pt}
\setlength{\itemsep}{0pt}
\setlength{\parsep}{0pt}
\setlength{\labelwidth}{\leftmargin}}
\item the end points of $\tilde{b}$ are projected to the same point by 
$\pi_{0}$, and the image loop $\pi_{0}(\tilde{b})$ is freely homotopic to 
$b_{0}$, 

\item the arc $\tilde{b}$ crosses $\tilde{a}_{0}$ once transversely, 

\item one of the end pints of $\tilde{b}$ is on $\partial D$ and the other lies 
outside of $D$, and 

\item the part $\tilde{b}':=\tilde{b} \setminus D$ is connected. 
\end{list}

We construct a marked once-holed torus $\tilde{X}=(\tilde{T},\tilde{\chi})$ 
belonging to $\mathfrak{T}_{a}[Y_{0}]$ as follows. We start with 
$D \cup \tilde{b}$. By identifying the end points of $\tilde{b}$ and thickening 
$\tilde{b}'$ slightly and appropriately, we obtain a once-holed torus 
$\tilde{T}$ so that $\pi_{0}$ induces a holomorphic mapping of $\tilde{T}$ into 
$R_{0}$. The curves $\tilde{a}_{0}$ and $\tilde{b}_{0}$ together make a mark 
$\tilde{\chi}$ of handle of $\tilde{T}$. It is obvious that 
$\tilde{X}=(\tilde{T},\tilde{\chi})$ is an element of 
$\mathfrak{T}_{a}[Y_{0}]$. As $D$ is a doubly connected domain on $\tilde{T}$ 
with $\Gamma(D) \subset \Gamma(\tilde{X})$, we have 
$$
\lambda_{a}[Y_{0}] \leqq \lambda(\tilde{X}) \leqq \lambda(D)=
\frac{1}{\,\pi\,}l(D)<\frac{1}{\,\pi\,}\{l(Y_{0})+\varepsilon\}.
$$
Since $\varepsilon$ is arbitrary, we deduce that 
$$
\lambda_{a}[Y_{0}] \leqq \frac{1}{\,\pi\,}l(Y_{0}),
$$
which completes the proof of Theorem~\ref{thm:critical_extremal:holomorphic}. 
\end{proof}

Next we evaluate the critical extremal length $\lambda_{c}[Y_{0}]$ for the 
existence of conformal mappings. The following theorem was announced in 
\cite{MasumotoPreprint}. 

\begin{thm}
\label{thm:critical_extremal:conformal}
$\lambda_{c}[Y_{0}]=\lambda(Y_{0})$. 
\end{thm}

\begin{proof}
If there is a conformal mapping $f$ of a marked once-holed torus $X$ into 
$Y_{0}$, then the image family $f_{*}\bigl(\Gamma(X)\bigr)$ is included in 
$\Gamma(Y_{0})$. Since conformal mappings keep extremal lengths invariant, it 
follows that $\lambda(X) \geqq \lambda(Y_{0})$ and hence that 
$$
\lambda_{c}[Y_{0}] \geqq \lambda(Y_{0}).
$$

To show that the sign of equality actually occurs, we employ results on 
Jenkins-Strebel differentials, that is, holomorphic quadratic differentials 
with closed trajectories (see Strebel \cite[Chapter~5]{Strebel1984}). There 
uniquely exists a doubly connected domain $\Delta_{0}$ on $R_{0}$ such that 
$\Gamma(\Delta_{0}) \subset \Gamma(Y_{0})$ and 
$\lambda(\Delta_{0})=\lambda(Y_{0})$. It is dense in $R_{0}$ and is swept out 
by closed horizontal trajectories of a holomorphic quadratic differential on 
$R_{0}$. Let $\delta_{0}=\pi/\lambda(Y_{0})$ and $r_{0}=e^{\delta_{0}}$, and 
set $A(\delta)=\{z \in \mathbb{C} \mid e^{-\delta}r_{0}<|z|<e^{\delta}r_{0}\}$ 
for $\delta \in (0,\delta_{0}]$. Then there is a conformal mapping $F_{0}$ of 
the annulus $A_{0}:=A(\delta_{0})$ onto $\Delta_{0}$, which is continuously 
extended to a union $E_{0}$ of open arcs on $\partial A_{0}$. We assume that 
$E_{0}$ is maximal with this property. Thus $F_{0}$ is a continuous mapping of 
$A_{0} \cup E_{0}$ {\em onto} $R_{0}$, and $R_{0}$ is obtained from 
$A_{0} \cup E_{0}$ by identifying points on $E_{0}$ in the obvious manner (see 
Jenkins-Suita \cite[Corollary~1 to Theorem~2]{JS}). Let $a_{0}'$ be the loop on 
$R_{0}$ corresponding to the circle $|z|=r_{0}$; we orient it so that it is 
freely homotopic to $a_{0}$. Take a piecewise analytic simple loop $b_{0}'$ on 
$R_{0}$ freely homotopic to $b_{0}$ such that the intersection of 
$F_{0}^{-1}(b_{0}')$ with the closure of a narrow annulus $A':=A(\delta_{1})$ 
is a radial segment, where $\delta_{1}$ is a sufficiently small positive 
number. By thickening $b_{0}'$ we obtain a doubly connected domain $B'$ with 
$b_{0}'$ separating the boundary components of $B'$. We choose $B'$ so that the 
union $T':=F_{0}(A') \cup B'$ is a once-holed torus included in $R_{0}$. 
Obviously, $\chi':=\{a_{0}',b_{0}'\}$ is a mark of handle of $T'$ and the 
inclusion mapping $T' \to R_{0}$ is a conformal mapping of the marked 
once-holed torus $(T',\chi')$ into $Y_{0}$. For each 
$\delta \in (0,\delta_{0})$ take a homeomorphism $h_{\delta}$ of the interval 
$[1,r_{0}^{2}]=[1,e^{\delta_{0}}r_{0}]$ onto itself such that 
$h_{\delta}(1)=1$, $h_{\delta}(e^{-\delta_{1}}r_{0})=e^{-\delta}r_{0}$ and 
$h_{\delta}(e^{\delta_{1}}r_{0})=e^{\delta}r_{0}$. It induces a homeomorphism 
$H_{\delta}$ of $R_{0}$ onto itself satisfying 
$H_{\delta} \circ F_{0}(re^{i\theta})=
F_{0}\bigl(h_{\delta}(r)e^{i\theta}\bigr)$. Intuitively, $H_{\delta}$ fattens 
$F_{0}(A')$ if $\delta>\delta_{1}$. The marked once-holed torus 
$X_{\delta}':=\bigl(H_{\delta}(T'),\chi_{\delta}'\bigr)$, where 
$\chi_{\delta}'=\{a_{0}',H_{\delta}(b_{0}')\}$, is conformally embedded into 
$Y_{0}$. Thus $X_{\delta}' \in \mathfrak{T}_{c}[Y_{0}]$. Since $H_{\delta}(T')$ 
includes $F_{0}\bigl(A(\delta)\bigr)$, we have 
$$
\lambda_{c}[Y_{0}] \leqq \lambda(X_{\delta}') \leqq \lambda\bigl(A(\delta)\bigr)
=\frac{\pi}{\,\delta\,}.
$$
Letting $\delta \to \delta_{0}$, we obtain 
$$
\lambda_{c}[Y_{0}] \leqq \frac{\pi}{\,\delta_{0}\,}=\lambda(Y_{0}).
$$
This completes the proof. 
\end{proof}

Theorems~\ref{thm:critical_extremal:holomorphic} 
and~\ref{thm:critical_extremal:conformal} give a simple alternative proof of 
one of our previous results. The next corollary implies that 
$\mathfrak{T}_{a}[Y_{0}] \setminus \mathfrak{T}_{c}[Y_{0}]$ has a nonempty 
interior since $\mathfrak{T}_{a}[Y_{0}]$ and $\mathfrak{T}_{a}[Y_{0}]$ are 
closed domains with Lipschitz boundary. 

\begin{cor}[\mbox{\cite[Theorem~8]{MasumotoPreprint}}]
\label{cor:compare:holomorphic&conformal}
$\lambda_{a}[Y_{0}]<\lambda_{c}[Y_{0}]$. 
\end{cor}

\begin{proof}
Let $\Delta_{0}$ be the doubly connected domain as in the proof of 
Theorem~\ref{thm:critical_extremal:conformal}. Then 
$$
\lambda_{a}[Y_{0}]=\frac{1}{\,\pi\,}l(Y_{0})<\frac{1}{\,\pi\,}l(\Delta_{0})=
\lambda(\Delta_{0})=\lambda(Y_{0})=\lambda_{c}[Y_{0}],
$$
where the inequality follows from the fact that $\Delta_{0}$ is a proper 
subdomain of $R_{0}$. 
\end{proof}

Let $\mathfrak{T}_{\infty}[Y_{0}]$ be the set of $X \in \mathfrak{T}$ such that 
there is a holomorphic mapping $f:X \to Y_{0}$ with $d(f)<\infty$. Again, 
$\Pi \circ \Sigma(\mathfrak{T}_{\infty}[Y_{0}])$ is a horizontal strip. In 
fact, there is a nonnegative number $\lambda_{\infty}[Y_{0}]$ such that 
\begin{list}
{{\rm (\roman{claim})}}{\usecounter{claim}
\setlength{\itemsep}{0pt}
\setlength{\parsep}{0pt}
\setlength{\labelwidth}{\leftmargin}}
\item if $\im\tau \geqq 1/\lambda_{\infty}[Y_{0}]$, then 
$X_{\tau}^{(s)} \not\in \mathfrak{T}_{\infty}[Y_{0}]$ for any $s \in [0,1)$, 
while 

\item if $\im\tau<1/\lambda_{\infty}[Y_{0}]$, then 
$X_{\tau}^{(s)} \in \mathfrak{T}_{\infty}[Y_{0}]$  for some $s \in [0,1)$ 
\end{list}
(see \cite[Theorem~4]{MasumotoPreprint}). 
Theorem~\ref{thm:critical_extremal:holomorphic} together 
with \cite[Theorem~2]{Masumoto2013} yields the following identity: 

\begin{cor}
\label{cor:holomorphic:finite}
$\lambda_{\infty}[Y_{0}]=\dfrac{1}{\,\pi\,}l(Y_{0})$. 
\end{cor}

Proposition~\ref{prop:holomorphic:critical_extremal} shows that the horizontal 
strip $\Pi \circ \Sigma(\mathfrak{T}_{a}[Y_{0}])$ never meets the critical 
horizontal line $\im\tau=1/\lambda_{a}[Y_{0}]$. In other words, if 
$\im\tau=1/\lambda_{a}[Y_{0}]$, then $X_{\tau}^{(s)}$ does not belong to 
$\mathfrak{T}_{a}[Y_{0}]$ for any $s \in [0,1)$ (see 
\cite[Theorem~6]{MasumotoPreprint}). This is not always the case for the 
critical extremal lengths for the existence of conformal mappings. In fact, if 
$Y_{0}$ is a marked once-holed torus, then the strip 
$\Pi \circ \Sigma(\mathfrak{T}_{c}[Y_{0}])$ and the line 
$\im\tau=1/\lambda_{c}[Y_{0}]$ intersect precisely at one point: there uniquely 
exists $\tau \in \mathbb{H}$ with $\im\tau=1/\lambda_{c}[Y_{0}]$ such that 
$X_{\tau}^{(s)}$ belongs to $\mathfrak{T}_{c}[Y_{0}]$ for some $s \in [0,1)$. 

We show that there is also a Riemann surface $Y_{0}$ such that 
$\Pi \circ \Sigma(\mathfrak{T}_{c}[Y_{0}])$ does not meet the horizontal line 
$\im\tau=1/\lambda_{c}[Y_{0}]$. To construct an example we give a preparatory 
consideration. Let $Y_{0}=(R_{0},\chi_{0})$, where 
$\chi_{0}=(a_{0},b_{0})$, be a Riemann surface with marked handle which is not 
a marked torus, and let $\Delta_{0}$ be the (unique) doubly connected domain on 
$R_{0}$ with $\Gamma(\Delta_{0}) \subset \Gamma(Y_{0})$ and 
$\lambda(\Delta_{0})=\lambda(Y_{0})$. Suppose that there is a conformal mapping 
$f$ of a marked once-holed torus $X_{\tau}^{(s)}$ with 
$\im\tau=1/\lambda_{c}[Y_{0}]$ into $Y_{0}$. Since 
$\lambda\bigl(f_{*}(D_{\tau})\bigr)=\lambda_{c}[Y_{0}]=\lambda(Y_{0})$, we 
obtain $f(D_{\tau})=\Delta_{0}$ by uniqueness. The horizontal arc 
$T_{\tau}^{(s)} \setminus D_{\tau}$ is mapped onto an arc $\gamma_{0}$ on the 
boundary $\partial\Delta_{0}$, and $f\bigl(T_{\tau}^{(s)}\bigr)$ is identical 
with $\Delta_{0} \cup \gamma_{0}$. This imposes a condition on $b_{0}$, for, it 
is freely homotopic to $f_{*}\bigl(b_{\tau}^{(s)}\bigr)$ on $R_{0}$. 

\begin{exmp}
\label{exmp:conformal:meet}
Set $R_{0}=T_{\tau_{0}} \setminus \pi_{\tau_{0}}(\{0,1/2\})$, which is a 
twice-punctured torus. Let $a_{0}$ and $b_{0}$ be the projections of the 
segments $[\tau_{0}/2,1+\tau_{0}/2]$ and $[3/4,3/4+\tau_{0}]$, respectively. 
They are simple loops on $R_{0}$, and make a mark $\chi_{0}$ of handle of 
$R_{0}$. Let $Y_{0}=(R_{0},\chi_{0})$. Then 
$\lambda_{c}[Y_{0}]=\lambda(Y_{0})=1/\im\tau_{0}$. Since 
$X_{\tau_{0}}^{(1/2)}$ belongs to $\mathfrak{T}_{c}[Y_{0}]$, the strip 
$\Pi \circ \Sigma(\mathfrak{T}_{c}[Y_{0}])$ and the line 
$\im\tau=1/\lambda_{c}[Y_{0}]$ meet at $\tau_{0}$. 
\end{exmp}

\begin{exmp}
\label{exmp:conformal:not_meet}
Let $R_{0}$ and $a_{0}$ be as in the preceding example. Let $b_{0}'$ be the 
projection of the polygonal arc obtained by joining the segments 
$[-1/4,\tau_{0}/4]$, $[\tau_{0}/4,1/2-\tau_{0}/4]$ and 
$[1/2-\tau_{0}/4,3/4+\tau_{0}]$. Set $\chi_{0}'=\{a_{0},b_{0}'\}$ and 
$Y_{0}'=(R_{0},\chi_{0}')$. Again, we have $\lambda_{c}[Y_{0}']=1/\im\tau_{0}$. 
However,  $\Pi \circ \Sigma(\mathfrak{T}_{c}[Y_{0}'])$ does not meet the 
critical horizontal line $\im\tau=1/\lambda_{c}[Y_{0}']$. 
\end{exmp}


\end{document}